\documentclass[letterpaper]{amsart}
\usepackage{amsmath,amsfonts,amsthm,amssymb}
\usepackage[all]{xy}
\usepackage{enumerate}
\usepackage{graphicx}
\usepackage{mathtools}
\usepackage{hyperref}
\usepackage{bbm}
\usepackage{xcolor}
\usepackage[T1]{fontenc}

\newcommand{\RR}{\mathbb{R}}

\newcommand{\dd}{\mathrm{d}}

\theoremstyle{definition}
\newtheorem{thm}{Theorem}[section]
\newtheorem{lma}[thm]{Lemma}
\newtheorem{cor}[thm]{Corollary}

\newtheorem{rmk}[thm]{Remark}

\allowdisplaybreaks

\title[Benamou-Brenier for multi-marginal optimal transport]{A Benamou-Brenier formulation for the multi-marginal optimal transport problem with infimal convolution cost}

\author{Friedemann Krannich}
\date{}
\address{Department of Mathematics, University of Toronto, Bahen Centre, 40 St. George St., Toronto, ON, M5S 2E4, Canada}
\email{f.krannich@mail.utoronto.ca}
\thanks{The author received support from the Department of Mathematics at the University of Toronto, the Canada Research Chairs program (Grant \#CRC-2020-00289) and the Natural Sciences and Engineering Research Council of Canada Discovery Grant (Grant \#RGPIN-2020-04162 and Grant \#RGPIN-2024-06235).}

\begin{document}
	
\maketitle

\begin{abstract} 
	We present a dynamical version for the multi-marginal optimal transport problem with infimal convolution cost, using the theory of Wasserstein barycentres. We show, how our formulation relates to the dynamical version of the multi-marginal optimal transport problem developed by Pass and Shenfeld.
\end{abstract}
	
\section{Introduction}
The classical optimal transport problem describes the minimal cost of transporting $\mu_1\in\mathcal P_p^{\text{ac}}(\RR^d)\coloneqq\big\{\mu\in\mathcal P(\RR^d): \mu<<\mathcal L^d,\int_{\RR^d}|x|^p\ \dd\mu(x)<\infty\big\}$ to $\mu_2\in\mathcal P_p^\text{ac}(\RR^d)$ through
\begin{gather*}
	W_p(\mu_1,\mu_2)\coloneqq \Big(\inf_{\gamma\in\Pi(\mu_1,\mu_2)}\int_{\RR^d\times\RR^d} |x_1-x_2|^p\ \dd\gamma(x_1,x_2)\Big)^{\frac{1}{p}}
\end{gather*}
where the set of transport plans $\Pi(\mu_1,\mu_2)$ are all couplings $\gamma\in\mathcal P(\RR^d\times\RR^d)$ that have marginals $\mu_1$ and $\mu_2$. In this manuscript, we are considering $p\in (1,\infty)$ fixed.
For an introduction to the theory of optimal transport and further literature we refer to the books \cite{villani2009,santambrogio2015}.
While $W_p(\mu_1,\mu_2)$ is a static problem, the celebrated publications by Benamou and Brenier \cite[Prop. 1.1]{benamou2000} (for $p=2$) and Brenier \cite[Th. 2.2]{brenier2004} (for general convex costs) recast it as a dynamical problem by showing that $W_p^p(\mu_1,\mu_2)$ is equal to
\begin{equation}\label{bb}
\begin{aligned}
	&BB(\mu_1,\mu_2)\coloneqq\inf_{(\rho,v)}\Big\{\int_0^1\int_{\RR^d}|v_t|^p \dd\rho_t\ \dd t:\\ &\partial_t\rho_t+\nabla\cdot (v_t\rho_t)=0,\ \rho_0=\mu_1,\ \rho_1=\mu_2\Big\}
\end{aligned}
\end{equation}
where $\rho_t\in\mathcal P_p(\RR^d)\coloneqq\big\{\mu\in\mathcal P(\RR^d):\int_{\RR^d}|x|^p\ \dd\mu(x)<\infty\big\}$, $v_t:\RR^d\rightarrow\RR^d$ for $t\in [0,1]$ and we require that $\int_0^1\int_{\RR^d}|v_t|\ \dd\rho_t\ \dd t<\infty$. The continuity equation and initial and terminal condition are understood in the weak sense, meaning that for each test function $\varphi\in C^1_c([0,1]\times\RR^d)$ we have
\begin{gather*}
	\int_0^1\int_{\RR^d}\big(\partial_t\varphi+\nabla_x\varphi\cdot v_t\big)\ \dd\rho_t\ \dd t=\int_{\RR^d}\varphi(1,x)\ \dd\mu_2(x)-\int_{\RR^d}\varphi(0,x)\ \dd\mu_1(x).
\end{gather*}
Given $\mu_1,\dots,\mu_N\in\mathcal P_p^\text{ac}(\RR^d)$, the multi-marginal optimal transport problem (see for example \cite{ruschendorf1981,gaffke1981,kellerer1984} for early references)
\begin{gather}\label{MMOT}
	C(\mu_1,\dots,\mu_N)\coloneqq\inf_{\gamma\in\Pi(\mu_1,\dots,\mu_N)}\int_{\RR^{dN}} c(x_1,\dots,x_N)\ \dd\gamma (x_1,\dots,x_N)
\end{gather}
generalizes $W_p^p(\mu_1,\mu_2)$ in a straightforward way, where $\Pi(\mu_1,\dots\mu_N)$ is the set of couplings $\mathcal P(\RR^{dN})$ with marginals $\mu_1,\dots,\mu_N$ and $c:\RR^{dN}\rightarrow\RR$ is a cost function. Existence and uniqueness of transport maps for $C(\mu_1,\dots,\mu_N)$ with quadratic ($p=2$) infimal convolution cost \eqref{infconvcost} was first studied by Gangbo and \'Swi\k ech \cite{gangbo1998}. For an introduction to the theory of multi-marginal optimal transport, see for example the survey \cite{pass2015}.\\
Until recently, no dynamical version in the spirit of \eqref{bb} for $C(\mu_1,\dots,\mu_N)$ was known. Pass and Shenfeld \cite[Th. 2.3]{pass2025} found a dynamical version for the multi-marginal optimal transport problem: They fix a reference coupling  $\omega\in\mathcal P(\RR^{dN})$ and then show, that if $c$ is convex and $\omega$-translation-invariant, i.e. that $c(x-\xi)=c(x)$ for all $x\in\RR^{dN}$ and $\omega$-a.e. $\xi\in\RR^{dN}$ \cite[Def. 1.1]{pass2025}, then $C(\mu_1,\dots,\mu_N)$ is equal to
\begin{equation}\label{psbb}
\begin{aligned}
	&PS(\mu_1,\dots,\mu_N)\coloneqq\inf_{(\gamma,V)}\Big\{\int_0^1\int_{\RR^{dN}}c(V_t(x))\ \dd\gamma_t(x)\ \dd t:\\ &\partial_t\gamma_t+\nabla\cdot (V_t\gamma_t)=0,\ \gamma_0=\omega,\ \pi^i_{\#}\gamma_1=\mu_i\ i\in[N]\Big\}.
\end{aligned}
\end{equation}
Compared to the classical Benamou-Brenier formula \eqref{bb}, this problem does not search for a flow of measures but rather for a flow of couplings, hence $\gamma_t\in\mathcal P(\RR^{dN})$ and $V_t:\RR^{dN}\rightarrow\RR^{dN}$ for $t\in[0,1]$ with the requirement that $\int_0^1\int_{\RR^{dN}}|V_t|\ \dd\gamma_t\ \dd t<\infty$. Here, $[N]\coloneqq\{1,\dots,N\}$ and $\pi^i:\RR^{dN}\rightarrow\RR^d$ denotes the projection on the $i$-th coordinate, given by $\pi^i(x_1,\dots,x_N)=x_i$ for $i\in [N]$. For a Borel map $T:\RR^{n}\rightarrow\RR^m$ we define the pushforward $T_{\#}\alpha=\beta$ for $\alpha\in\mathcal P(\RR^{n})$ and $\beta\in\mathcal P(\RR^m)$ through
\begin{gather*}
	\alpha(T^{-1}(A))=\beta(A)
\end{gather*}
for all $A\subset\RR^m$ Borel.\\
If the cost function $c$ in \eqref{MMOT} is the infimal convolution cost
\begin{gather}\label{infconvcost}
	c(x_1,\dots,x_N)\coloneqq\inf_{z\in\RR^d}\sum_{i=1}^N|x_i-z|^p,
\end{gather}
Chiappori, McCann and Nesheim \cite[Sect. 5]{chiappori2010} and Carlier and Ekeland \cite[Prop. 3]{carlier2010} showed, that $C(\mu_1,\dots,\mu_N)$ is equal to
\begin{gather}\label{WBC}
	WB(\mu_1,\dots,\mu_N)\coloneqq\inf_{\nu\in\mathcal{P}_p(\RR^d)}\sum_{i=1}^NW_p^p(\nu,\mu_i).
\end{gather}
This problem is called the $p$-Wasserstein barycentre problem \cite{agueh2011} and we call its solution $\bar\nu$ the $p$-barycentre of $\mu_1,\dots,\mu_N$. Agueh and Carlier \cite[Prop. 3.5]{agueh2011} (for $p=2$) and Brizzi, Friesecke and Ried \cite[Theo. 1.4]{brizzi2025} (for $p\in (1,\infty)$) showed, that the $p$-barycentre is uniquely attained and that $\bar\nu\in\mathcal P_p^{\text{ac}}(\RR^d)$. \\
Using the theory of $p$-Wasserstein barycentres, we propose another dynamical formulation for $C(\mu_1,\dots,\mu_N)$, that we found simultaneously and independently from the formulation of Pass and Shenfeld: In Th. \ref{mmoteqaultommbb} we show equality of $C(\mu_1,\dots,\mu_N)$ with cost function \eqref{infconvcost} and
\begin{equation} \label{MMBB}
\begin{aligned}
	&DC(\mu_1,\dots,\mu_N)\coloneqq
	\inf_{(\rho^i,v^i)_{i\in[N]}}\Big\{\int_0^1\sum_{i=1}^N\int_{\RR^d}|v^i_t|^p\ \dd\rho^i_t\ \dd t:\\ &\partial_t\rho^i_t+\nabla\cdot (v^i_t\rho^i_t)=0,\ \rho^1_0=\dots=\rho^N_0,\ \rho^i_1=\mu_i\Big\}.
\end{aligned}
\end{equation}
In our dynamical formulation, we search for $N$ flows of probability measures, with a coupled initial constraint and fixed endpoints, hence $\rho_t^i\in\mathcal P_p(\RR^d)$, $v_t^i:\RR^d\rightarrow\RR^d$ for $t\in [0,1]$ and we require that $\int_0^1\int_{\RR^d}|v_t^i|\ \dd\rho_t^i\ \dd t<\infty$ for $i\in [N]$.
One notices that, using the insight of Benamou and Brenier \cite{benamou2000}, we can change variables $(\rho^i,v^i)\rightarrow (\rho^i,m^i)\coloneqq (\rho^i,v^i\rho^i)$ to recast $DC(\mu_1,\dots,\mu_N)$ as a convex optimization problem with linear constraints.
Our dynamical problem works without a reference coupling $\omega$ and has a natural physical interpretation for $p=2$ (see Rmk. \ref{physint}). We show in Th. \ref{psequalitythm}, how to construct a solution $(\bar\gamma, \bar V)$ to $PS(\mu_1,\dots,\mu_N)$ from a solution $(\bar \rho^i,\bar v^i)_{i\in [N]}$ to $DC(\mu_1,\dots,\mu_N)$, if the reference coupling is given through $\omega=(\text{id}\times\dots\times\text{id})_{\#}\bar\nu$.
\subsection{Plan of the manuscript}
In section 2, we show\\
$C(\mu_1,\dots,\mu_N)=DC(\mu_1,\dots,\mu_N)$, using the theory of $p$-Wasserstein barycentres. In section 3, we construct a solution to $PS(\mu_1,\dots,\mu_N)$ from a solution to\\ $DC(\mu_1,\dots,\mu_N)$. We conclude the manuscript in section 4.
\section{Dynamical multi-marginal optimal transport through $p$-Wasserstein barycentres}
For the rest of this manuscript, we consider the multi-marginal optimal transport problem $C(\mu_1,\dots,\mu_N)$ with cost function given through the infimal convolution cost \eqref{infconvcost} with $p\in (1,\infty)$ fixed.
Using \eqref{WBC}, we find a new dynamical formulation for $C(\mu_1,\dots,\mu_N)$:
\begin{thm}\label{mmoteqaultommbb}
	$C(\mu_1,\dots,\mu_N)=DC(\mu_1,\dots,\mu_N)$ holds with cost function given by \eqref{infconvcost}.
\end{thm}
\begin{proof}
	From \cite[Prop. 3]{carlier2010} we get $C(\mu_1,\dots,\mu_N)=WB(\mu_1,\dots,\mu_N)$ and hence,
	using \eqref{bb} for each term $W_p^p(\nu,\mu_i)$, we write $WB(\mu_1,\dots,\mu_N)$ as
	\begin{align*}
		&\inf_{\nu\in\mathcal{P}_p(\RR^d)}\sum_{i=1}^NW_p^p(\nu,\mu_i)\\
		&=\inf_{\nu\in\mathcal{P}_p(\RR^d)}\sum_{i=1}^N\inf_{(\rho^i,v^i)}\Big\{\int_0^1\int_{\RR^d}|v^i_t|^p\ \dd\rho^i_t\ \dd t:\ \partial_t\rho^i_t+\nabla\cdot (v^i_t\rho^i_t)=0, \rho^i_0=\nu, \rho^i_1=\mu_i\Big\}\\
		&=\inf_{(\rho^i,v^i)_{i\in [N]}}\Big\{\int_0^1\sum_{i=1}^N\int_{\RR^d}|v^i_t|^p \dd\rho^i_t \dd t: \partial_t\rho^i_t+\nabla\cdot (v^i_t\rho^i_t)=0, \rho^1_0=\dots=\rho^N_0, \rho^i_1=\mu_i \Big\}.
	\end{align*}
\end{proof}
The dynamical version \eqref{bb} of $W_p^p(\bar\nu,\mu_i)$ allows for an explicit description of the minimizers $(\bar\rho^i,\bar v^i)_{i\in [N]}$ in $DC(\mu_1,\dots,\mu_N)$ (see \cite[Sect. 5]{santambrogio2015} for details):
Let $T^i$ be the optimal transport map in $W_p(\bar\nu,\mu_i)$ that transports $\bar{\nu}$ to $\mu_i$. For $t\in [0,1]$ and $y\in\text{supp}(\bar\nu)$, we define the map 
\begin{gather*}
	T^i_t(y)\coloneqq (1-t)y+t T^i(y),
\end{gather*}
which is invertible for all $t\in [0,1]$. Then, McCann's displacement interpolant \cite[Def. 1.1]{mccann1997} $\bar\rho_t^i\coloneqq (T^i_{t})_{\#}\bar\nu$
together with the implicitly defined vector field
\begin{gather*}
	\bar v^i_t(T^i_t(y))\coloneqq T^i(y)-y
\end{gather*}
are the minimizers in $BB(\bar\nu,\mu_i)$ for every $i\in [N]$ \cite[Prop. 5.30]{santambrogio2015} and hence $(\bar\rho^i,\bar v^i)_{i\in[N]}$ are the minimizers in $DC(\mu_1,\dots,\mu_N)$.\\
We define $c_p(x)\coloneqq |x|^p$ and for a function $\psi$ we denote
\begin{gather*}
	\psi^{c_p}(x)\coloneqq\inf_{y\in\RR^d}|x-y|^p-\psi(y).
\end{gather*}
The next Lemma is a straightforward extension of \cite[Prop. 3.8]{agueh2011}, but we give its proof for completeness.
\begin{lma}\label{velocitysumlma}
	We have 
	\begin{gather*}
		\sum_{i=1}^N \nabla c_p(T^i(y)-y)=0
	\end{gather*}
	for a.e. $y\in\text{supp}(\bar \nu)$.
\end{lma}
\begin{proof}
	$C(\mu_1,\dots,\mu_N)$ has the dual problem \cite[Theo. 4.1]{brizzi2025}
	\begin{gather*}
		\sup_{(\varphi_i)_{i\in [N]}}\Big\{\sum_{i=1}^N\int_{\RR^d}\varphi_i\ \dd\mu_i:\ \varphi_i\in L^1(\mu_i),\ \varphi_1(x_1)+\dots+\varphi_N(x_N)\leq c(x_1,\dots,x_N)\Big\}
	\end{gather*}
	with maximizers $(\bar\varphi_i)_{i\in [N]}$ given through $\bar\varphi_i=\psi^{c_p}_i$ with $\sum_{i=1}^N\psi_i(y)=0$ for\\ $y\in\text{supp}(\bar\nu)$. We have
	\begin{gather*}
		W_p^p(\bar\nu,\mu_i)=\int_{\RR^d}\psi_i(y)\ \dd\bar\nu(y)+\int_{\RR^d}\psi^{c_p}_i(x)\ \dd\mu_i(x)
	\end{gather*}
	and therefore we have for a.e. $(y,x)=(y,T^i(y))\in\text{supp}(\bar\gamma_i)$ ($\bar\gamma_i$ is the optimal coupling in $W_p(\bar\nu,\mu_i)$)
	\begin{gather*}
		\psi^{c_pc_p}_i(y)\geq\psi_i(y)= c_p(x-y)-\psi^{c_p}_i(x)\geq\psi^{c_pc_p}_i(y).
	\end{gather*}
	Therefore, we get
	\begin{gather*}
		\sum_{i=1}^N\psi^{c_pc_p}_i(y)\geq\sum_{i=1}^N\psi_i(y)= 0
	\end{gather*}
	with equality a.e. $y\in\text{supp}(\bar\nu)$. Gangbo and McCann \cite[Th. 3.3]{gangbo1996} showed $\bar\nu$-a.e. differentiability of $\psi_i$, hence we get for a.e. $y\in\text{supp}(\bar\nu)$ that
	\begin{gather*}
		0=\sum_{i=1}^N\nabla\psi^{c_pc_p}_i(y)=\sum_{i=1}^N-\nabla c_p(T^i(y)-y).
	\end{gather*}
\end{proof}
\begin{cor}\label{velocitycor}
	Let $(\bar\rho^i,\bar v^i)_{i\in [N]}$ be the optimizers for $DC(\mu_1,\dots,\mu_N)$ with $p=2$. Let $\bar m^i\coloneqq \bar v^i\bar\rho^i$ be the momentum variables, then we have for a.e. $y\in\text{supp}(\bar\nu)$
	\begin{gather*}
		\sum_{i=1}^N\bar v^i_0(y)=0=\sum_{i=1}^N\bar m^i_0(y).
	\end{gather*}
\end{cor}
\begin{proof}
	We have
	\begin{gather*}
		\sum_{i=1}^N\bar v^i_0(y)=\sum_{i=1}^NT^i(y)-y=\sum_{i=1}^N\nabla c_2(T^i(y)-y)=0
	\end{gather*}
	due to Lemma \ref{velocitysumlma}. Noticing that $\bar m^i_0(y)=\bar v^i_0(y)\bar \rho^i_0(y)=\bar v^i_0(y)\bar\nu (y)$, the second equality follows.
\end{proof}
\begin{rmk}\label{physint}
	Our dynamical multi-marginal optimal transport problem\\ $DC(\mu_1,\dots,\mu_N)$ is inspired  by a physical interpretation for $p=2$, that is based on Corollary \ref{velocitycor}: Let $N=3$ and consider a table with 3 holes randomly spaced over the table. Through each hole hangs a rope, having a ball with unit weight on its lower end, each ball representing a measure $\mu_i$. The upper ends of the ropes are all connected to another ball with unit weight that lays on the table. The three balls that are hanging pull the ball on the table around, until equilibrium is reached. Then the ball on the table represents the $2$-barycentre $\bar\nu$ and the velocities of the three ropes at the $2$-barycentre sum to zero, causing the ball on the table to stop moving.
\end{rmk}

\section{Relation between $PS(\mu_1,\dots,\mu_N)$ and $DC(\mu_1,\dots,\mu_N)$}
Next we show, how $PS(\mu_1,\dots,\mu_N)$ and $DC(\mu_1,\dots,\mu_N)$ are related.
\begin{thm}\label{psequalitythm}
	Let $\omega=(\text{id}\times\dots\times\text{id})_{\#}\bar\nu\in\mathcal P(\RR^{dN})$, then for the cost function \eqref{infconvcost} we have $PS(\mu_1,\dots,\mu_N)=DC(\mu_1,\dots,\mu_N)$ and we can construct minimizers $(\bar\gamma,\bar V)$ for $PS(\mu_1,\dots,\mu_N)$ from minimizers $(\bar\rho^i,\bar v^i)_{i\in [N]}$ for $DC(\mu_1,\dots,\mu_N)$.
\end{thm}
\begin{proof}
	The minimizer in $c(x_1,\dots,x_N)$ is the unique $\bar z\in\RR^d$ characterized through the condition
	\begin{gather}\label{metbaroptcond}
		0=\sum_{i=1}^N\nabla c_p(x_i-\bar z).
	\end{gather}
	Hence, if $\bar z$ is the minimizer for $c(x_1,\dots,x_N)$ then $\bar z-\xi$ is the minimizer for $c(x_1-\xi,\dots,x_N-\xi)$,
	which shows that $c$ is $\omega$-translation-invariant, illustrating \cite[Ex. 1.4]{pass2025}. Hence, \cite[Th. 2.3]{pass2025} gives $PS(\mu_1,\dots,\mu_N)=C(\mu_1,\dots,\mu_N)$ and Th. \ref{mmoteqaultommbb} gives $C(\mu_1,\dots,\mu_N)=DC(\mu_1,\dots,\mu_N)$.\\
	We use $(\bar\rho^i,\bar v^i)$ and the interpolations $T^i_t$ defined in section 2, recalling that\\ $(\bar\rho^i,\bar v^i)_{i\in [N]}$ constructed in this way minimize $DC(\mu_1,\dots,\mu_N)$. We define the interpolant
	\begin{gather*}
		S_t(x)\coloneqq (T^1_t(x_1),\dots,T^N_t(x_N))
	\end{gather*}
	where $(x_1,\dots,x_N)=x\in\text{supp}(\omega)$. We consider the time-dependent coupling $\gamma_t\coloneqq (S_t)_{\#}\omega$
	and we define implicitly for $x\in\text{supp}(\omega)$
	\begin{gather*}
		V_t(S_t(x))\coloneqq (\bar v^1_t(T^1_t(x_1)),\dots,\bar v^N_t(T^N_t(x_N)))
	\end{gather*}
	so that
	\begin{gather*}
		\begin{cases}
			\partial_tS_t(x)=V_t(S_t(x))\\
			S_0(x)=x
		\end{cases}
	\end{gather*}
	hold.
	Then by construction $(\gamma,V)$ fulfill the continuity equation and\\ $\pi^i_{\#}\gamma_1=(T^i_1)_{\#}\bar\nu=\mu_i$.
	Hence,
	\begin{align*}
		&\int_{\RR^{dN}}c(V_t(x))\ \dd\gamma_t(x)\\
		&=
		\int_{\RR^{d}}\inf_{z\in\RR^d}\sum_{i=1}^N|\bar v^i_t(T^i_t(y))-z|^p\ \dd\bar\nu(y)\\
		&=\int_{\RR^{d}}\inf_{z\in\RR^d}\sum_{i=1}^N|T^i(y)-y-z|^p\ \dd\bar\nu(y)
	\end{align*}
	and by Lemma \ref{velocitysumlma} and \eqref{metbaroptcond} we get that for a.e. $y\in\text{supp}(\bar\nu)$
	\begin{gather*}
		\inf_{z\in\RR^d}\sum_{i=1}^N|T^i(y)-y-z|^p=\sum_{i=1}^N|T^i(y)-y|^p.
	\end{gather*}
	Therefore,
	\begin{align*}
		PS(\mu_1,\dots,\mu_N)&\leq\int_0^1\int_{\RR^{dN}}c(V_t(x))\ \dd\gamma_t(x)\ \dd t\\
		&=\int_0^1\int_{\RR^{d}}\inf_{z\in\RR^d}\sum_{i=1}^N|T^i(y)-y-z|^p\ \dd\bar\nu(y)\ \dd t\\
		&=\int_0^1\int_{\RR^{d}}\sum_{i=1}^N|T^i(y)-y|^p\ \dd\bar\nu(y)\ \dd t\\
		&=\int_0^1\sum_{i=1}^N\int_{\RR^{d}}|\bar v^i_t|^p\ \dd\bar \rho^i_t\ \dd t\\
		&=DC(\mu_1,\dots,\mu_N)
	\end{align*}
	and hence $(\gamma,V)$ are minimizers for $PS(\mu_1,\dots,\mu_N)$.
\end{proof}
\section{Conclusion}
In this manuscript, we present a new dynamical formulation $DC(\mu_1,\dots,\mu_N)$ for the multi-marginal optimal transport problem.
Compared to the dynamical formulation $PS(\mu_1,\dots,\mu_N)$ by Pass and Shenfeld \cite{pass2025}, our formulation does not need a reference coupling $\omega$. However, $PS(\mu_1,\dots,\mu_N)$ allows for a more general class of cost functions. Especially Th. \ref{psequalitythm} only works for the infimal convolution cost \eqref{infconvcost}, as it needs the specific characterization \eqref{metbaroptcond} of the minimizer $\bar z$ in the cost function.\\
One of the applications of the Benamou-Brenier formula \eqref{bb} was, that it allowed for new algorithms for the computation of $W_p(\mu_1,\mu_2)$. Pass and Shenfeld \cite[Sect. 4]{pass2025} also use their dynamical formulation for numerical computations. We conjecture, that $DC(\mu_1,\dots,\mu_N)$ gives rise to more efficient algorithms, as it allows to solve a de-coupled system of $N$ continuity equations on $\RR^d$ with coupled initial constraints, while a numerical version of $PS(\mu_1,\dots,\mu_N)$ solves for one continuity equation on $\RR^{dN}$.
\section*{Conflicts of interest}

The author does not work for, advises, owns shares in, or receives funds from any organization
that could benefit from this article, and has declared no affiliations other than their research
organizations.

\section*{Acknowledgements}

The author is grateful to Robert McCann and Adrian Nachman for their supervision and ongoing support while writing this manuscript.

\bibliographystyle{plain}
\bibliography{bibliography}

\end{document}